\documentclass[8pt]{article} % use larger type; default would be 10pt

\usepackage[utf8]{inputenc} % set input encoding (not needed with XeLaTeX)

\usepackage{geometry} % to change the page dimensions
\usepackage{graphicx} % support the \includegraphics command and options

\usepackage{booktabs} % for much better looking tables
\usepackage{array} % for better arrays (eg matrices) in maths
\usepackage{paralist} % very flexible & customisable lists (eg. enumerate/itemize, etc.)
\usepackage{verbatim} % adds environment for commenting out blocks of text & for better verbatim
\usepackage{subfig} % make it possible to include more than one captioned figure/table in a single float
\usepackage{amsmath}
\usepackage{amssymb}
\usepackage{amsthm}
\usepackage{url}

\newtheorem{theorem}{Theorem}
\newtheorem{proposition}{Proposition}
\newtheorem{lemma}{Lemma}
\newtheorem{definition}{Definition}
\newtheorem{corollary}{Corollary}

\usepackage{fancyhdr} % This should be set AFTER setting up the page geometry
\usepackage{sectsty}
\allsectionsfont{\sffamily\mdseries\upshape} % (See the fntguide.pdf for font help)
\usepackage[nottoc,notlof,notlot]{tocbibind} % Put the bibliography in the ToC
\usepackage[titles,subfigure]{tocloft} % Alter the style of the Table of Contents

\title{A Universal Quaternary Quadratic Form over Gaussian Integers}
\author{Felix Sidokhine}
\begin{document}
\maketitle

\section{Introduction}

The earliest study of universal forms should be credited to Lagrange with his famous four-square theorem over the integers. This result has been since revisited and re-proved using other approaches than Lagrange's original method. As pointed out in J. Conway's paper ``Universal Quadratic Forms and the Fifteen Theorem" \cite{Conway}, many number theoretists including Ramanujan have expressed interest in identifying universal quadratic forms. By using different techniques they have shown many results about quadratic forms over $\mathbb{Z}$.
\newline
\newline
\noindent The earliest study of this question could be attributed to L.J. Mordell \cite{Mordell1920}, however his work was focused on real fields. The first  attempt at studying quadratic forms over imaginary algebraic integers, is perhaps Niven's work about the representation of certain Gaussian integers as sums of two squares \cite{Niven1940}. However, imaginary fields proved themselves to be a very tough challenges, and I have been unable to track down any scholarly articles (or otherwise) pertaining to universal forms over them. 
\newline
\newline
\noindent On the other hand, real quadratic fields have shown themselves to be more docile, at least in the case of ternary forms. ``Ternary universal integral quadratic forms over real quadratic fields" \cite{Wai-kiuChan1996} by Chan, Kim and Raghavan contain a great overview of existing results and techniques pertaining to the problem. For the case of quaternary quadratic forms over such fields, Clark, Hicks, Thompson and Walters \cite{PeteL.Clark} have completed a very complete research paper using techniques from the geometry of numbers.
\newline
\newline
\noindent In this paper, I decided to re-visit the challenge of finding a universal form over $\mathbb{Z}[i]$ and I have shown that in fact at least one exists. By ``universal form", we mean a form $f$ with gaussian integer coefficients in gaussian integer variables which represent all gaussian integers \cite{Dickson1929}. Indeed, I speculate that only my use of purely algebraic arguments has allowed me to do so and it is very unlikely any other methods would have led to the desired results. Moreover, this work required me to think carefully about the structure of $\mathbb{Z}[i]$ and how to circumvent the problems caused by the unit group instead of trying to adapt to it.

\section{About Gaussian Integers}

While $\mathbb{Z}[i]$ is known to be a unique factorization domain, and even Euclidean, it is extremely difficult to reproduce any of the arguments used in the case of $\mathbb{Z}$. Many of these limitations come from two places: the unit group and the lack of well-ordering.
\newline
\newline
\noindent My first challenge is to eliminate the problems associated with the presence of the unit group. To do so, I first claim that any $z \in \mathbb{Z}[i]$ can be expressed as:

\begin{equation}
z = i^s z'
\end{equation} 

\noindent where $\frac{-\pi}{4} < Arg(z') \leq \frac{\pi}{4}$. One can verify this statement directly. In turn I can claim that the zone $(-\frac{\pi}{4},\frac{\pi}{4}]$ contains the set of prime elements $\Pi' \subset \Pi_{\mathbb{Z}[i]}$ which express uniquely any $z \in \mathbb{Z}[i]$. This ``uniqueness" is no longer the uniqueness as understood in algebraic number theory but a perfect copy of the uniqueness of factorization that holds in $\mathbb{Z}$. We shall call $\Pi'$ the set of canonical prime elements of $\mathbb{Z}[i]$.
\newline
\newline
\noindent This set $\Pi'$ can also be well-ordered (this also implies that any of the subsets of $\Pi'$ can be well-ordered, a fact that we will make use of later). This is done as follows: let $p$ and $p'$ be elements of $\Pi'$. $p$ preceeds $p'$ if $N(p) < N(p')$. Should their norms be equal, $p$ preceeds $p'$ if $\Im(p) < \Im(p')$. One can directly verify that this is indeed a well-ordering of $\Pi'$. I shall now introduce an additional mapping that will be used throughout this work:

\begin{definition}
Let $\nu : \mathbb{Z}[i] \to \mathbb{Z}^+ $ be defined as follows: Given $z = i^s p_1^{\alpha_1}...p_k^{\alpha_k}$ where $p_t \in \Pi'$, $\nu(z) = \sum_{t = 1}^k \alpha_t$. 
\end{definition}

One can immidiately notice several properties of $\nu$:

\begin{itemize}
\item $\nu(a) = 0 \iff a \in U$\\
\item $\nu(a) = 1 \iff a \in \Pi'$ \\
\item $\nu(ab) = \nu(a) + \nu(b)$ \\
\item If $a | b \land b \nmid a \Rightarrow \nu(a) < \nu(b)$
\end{itemize}

\section{Some Preliminary Results}

From here on, we shall only use canonical prime elements. Using norms, we shall partition the set of all canonical primes $\Pi'$ as follows: 

\begin{proposition}
$\Pi' = A \cup B \cup C$, where $A = \{ p \in \Pi' | N(p) \equiv 1 \mod 8 \}$, $B = \{1+i\}$, and $C = \{ p \in \Pi' | N(p) \equiv 5 \mod 8 \}$
\end{proposition}

\begin{lemma}\label{lemma2}
The equation $x^2 + iy^2 = pz$, $(x,y)=1$ is solvable if and only if $p$ belongs to $A \cup B$.
\end{lemma}

\begin{proof}
Let us look at the following modular equation:
\begin{equation}
x^2 + iy^2 \equiv 0 \mod p
\end{equation} 

\noindent The solvability of the above equation is equivalent to $i$ being a quadratic residue modulo $p$.
\newline
\newline
\noindent Let us show that if a solution existed, then $p$ was an element of $A$. Let $x_0$ be a solution to the congruence $x_0^2 \equiv i \mod p$. Then:
\begin{equation}
(x_0^2)^\frac{N(p)-1}{2} \equiv i^\frac{N(p)-1}{2} \mod p
\end{equation}

\noindent However by Fermat's little theorem, $(x_0^2)^\frac{N(p)-1}{2} \equiv 1 \mod p$, which in turn implies:

\begin{equation}
 i^\frac{N(p)-1}{2} \equiv 1 \mod p
\end{equation}

\noindent which is only possible if $N(p) \equiv 1 \mod 8$. Which concludes the first part of the proof.
\newline
\newline
\noindent Let $p \in A$, then $x^2 - i \equiv 0 \mod p$ is solvable. Since $\mathbb{Z}[i]/(p)$ is a field, we can apply the same arguments as those used to identify quadratic residues in $\mathbb{Z}$, except amended by Fermat's little theorem for Gaussian integers \cite{HarryPollard1985}  to obtain the result.
\end{proof}

\begin{lemma}\label{lemma2}
The equation $x^2 + iy^2 = pz$, $(x,y)=1$ where $p \in A$ has a solution $(x_0,y_0,z_0)$ such that $z_0$ has only prime factors that preceed $p$.
\end{lemma}

\begin{proof}

Let $(X_0,Y_0,Z_0)$ be some arbitrary solution of the above equation. Then we can use the euclidean property of $\mathbb{Z}[i]$ to perform division with remainder, such that:
\begin{equation}
\begin{cases}
X_0 \equiv r_{x_0} \mod p \text{ where } N(r_{x_0}) < \frac{N(p)}{2} \\
Y_0 \equiv r_{y_0} \mod p \text{ where } N(r_{y_0}) < \frac{N(p)}{2}
\end{cases}
\end{equation}

\noindent The above implies that $r_{x_0}$ and $r_{y_0}$ are a solution of $x^2 + iy^2 = pz$ with a corresponding new $z_0'$. Should they not be relatively prime, we reduce by their greatest common divisor. By using the triangle inequality for $\mathbb{C}$ we obtain:

\begin{equation}\footnote{$||z|| = \sqrt{N(z)}$}
N(p) > ||r_{x_0}||^2  + ||r_{y_0}||^2 \geq ||r_{x_0}^2  + ir_{y_0}^2|| = ||pz_0'||
\end{equation}

\noindent Which yields $N(p) > N(z_0')$, from which follows that any prime divisor of $z_0'$ preceeds $p$.
\end{proof}

\begin{theorem}\label{prop1}
The equation $x^2 + iy^2 = p$ is solvable for any $p \in A$.
\end{theorem}

\begin{proof}

$A$ is well-ordered by the considerations we outlined in the section 2. Suppose that theorem \ref{prop1} holds for all $p_1,...,p_n$ from $A$ and does not hold for $p_{n+1}$.
\newline
\newline
\noindent Suppose we found a solution $(X_0,Y_0,Z_0)$ such that $X_0^2 + iY_0^2 = p_{n+1}Z_0$. By lemma \ref{lemma2} we can reduce this solution to a solution $(x_0,y_0,z_0)$ such that all the prime factors of $z_0$ preceed $p_{n+1}$. Say $z_0 = p_1^{\alpha_1}...p_k^{\alpha_k}$, where every $p_i$ preceeds $p_{n+1}$ and let $\nu(z_0) = m$. Let $p'$ be any prime divisor of $z_0$, then we can claim:

\begin{equation}
x_0'^2 + iy_0'^2 = p'
\end{equation}

\noindent which in turn yield the following simultaneous congruences:

\begin{equation}
\begin{cases}
x_0^2 \equiv -iy_0^2  \mod p' \\
x_0'^2 \equiv -iy_0'^2 \mod p'
\end{cases}
\end{equation}

\noindent By multiplying the congruences we obtain:

\begin{equation}
(x_0x_0')^2 + (y_0y_0')^2 \equiv 0 \mod p'
\end{equation}

\noindent Which in turn we can factor and obtain:

\begin{equation}
(x_0x_0' + iy_0y_0')(x_0x_0' - iy_0y_0') \equiv 0 \mod p'
\end{equation}

\noindent This last equation tells us that at least one of the factors must have been divisble by $p'$. Let us now return to the original equation and factor $z_0 = p'z_0'$ to reflect this fact:

\begin{equation}
x_0^2 + iy_0^2 = p_{n+1}p'z_0'
\end{equation}

\noindent now let us multiply the above equation through by $p'$:

\begin{equation}
(x_0^2 + iy_0^2)p' = p_{n+1}p'^2z_0'
\end{equation}

\noindent However this equation factors as follows:

\begin{equation}
(x_0x_0' \mp iy_0y_0')^2 + i(x_0y_0' \pm x_0'y_0)^2 = p_{n+1} p'^2 z_0'
\end{equation}

\noindent and by the indentity in (11) we can conclude that a reduction by $d$ ($d \equiv 0 \mod p'$ and $(d,p_{n+1})=1$) is possible, leading us to the following substitution:

\begin{equation}
\begin{cases}
x_0^1 = \frac{x_0x_0' \mp iy_0y_0'}{d} \\
y_0^1 = \frac{x_0y_0' \pm x_0'y_0}{d}
\end{cases}
\end{equation}

\noindent which in turn satisfy:

\begin{equation}
(x_0^1)^2 + i(y_0^1)^2 = p_{n+1}z_0^1 \text{ and } (x_0^1,y_0^1)=1
\end{equation}

\noindent Since we reduced by $d$, $\nu(z_0^1) < \nu(z_0)$, where $z_0^1 | z_0$. As the reduction can be repeated, we have that $\nu(z_0^k) < \nu(z_0^{k-1}) < ... < \nu(z_0)$, where $z_0^k | z_0$; and therefore after a maximum of $m$ reductions, $\nu(z_0^m) = 0$, concluding the proof.
\end{proof}

\begin{theorem}[Niven-Mordell Theorem]
If $a$ and $b$ are rational integers, a representation $x^2 + y^2 = a + 2ib$ in Gaussian integers is possible if and only if not both $a \equiv 2 (\mod 4), b \equiv 1 (\mod 2)$ are satisfied.
\end{theorem}

\begin{theorem}\label{niven1}
Let $p \in C$, then there exist $x_0,y_0 \in \mathbb{Z}[i]$ such that $p = x_0^2 + y_0^2$ or $p = i(x_0^2 + y_0^2)$.
\end{theorem}

\begin{proof}
Theorem \ref{niven1} is a direct consequence of the Niven-Mordell theorem. Say $N(p) = q$ where $q$ is obviously prime and $q = x^2 + y^2$ where $x$ and $y$ are unique up-to signs and permutations, a consequence of Fermat's fundamental theorem. Since $q \equiv 5 \mod 8$, $x = a$ and $y = 2b$ where $(a,2b)=1$ and $ab \equiv 1 \mod 2$ (another possibility was $x = 2b$ and $y= a$ subject to the same constraints). We can therefore conclude that $p$ was actually $p = a + 2bi$ (or $2b + ai$) and therefore the Niven-Mordell theorem applies, which concludes the proof of theorem \ref{niven1}.
\end{proof}

\section{$x^2 + iy^2 + w^2 + iz^2$ is universal over Gaussian Integers}

\begin{lemma}
Given a commutative ring $R$ and the expressions $Ax_1^2 + By_1^2 + Az_1^2 + Bw_1^2$, $Ax_2^2 + By_2^2 + Az_2^2 + Bw_2^2$, their product is $X^2 + (AB)Y^2 + Z^2 + (AB)W^2$, where
\begin{equation}
\begin{cases}
X = Ax_1x_2 + By_1y_2 + Az_1z_2 + Bw_1w_2 \\
Y = x_1y_2 - y_1x_2 - z_1w_2 + w_1z_2 \\
Z = Ax_1z_2 + By_1w_2 - Az_1x_2 - Bw_1y_2 \\
W = x_1w_2 - y_1z_2 + z_1y_2 - w_1x_2
\end{cases}
\end{equation}
\end{lemma}

\begin{corollary}
Over $\mathbb{Z}[i]$ the product of the expressions $x_1^2 + iy_1^2 + z_1^2 + iw_1^2$ and $x_2^2 + iy_2^2 + z_2^2 + iw_2^2$ is $X^2 + iY^2 + Z^2 + iW^2$ where,
\begin{equation}
\begin{cases}
X = x_1x_2 + iy_1y_2 + z_1z_2 + iw_1w_2 \\
Y = x_1y_2 - y_1x_2 - z_1w_2 + w_1z_2 \\
Z = x_1z_2 + iy_1w_2 - z_1x_2 - iw_1y_2 \\
W = x_1w_2 - y_1z_2 + z_1y_2 - w_1x_2
\end{cases}
\end{equation}
\end{corollary}

\begin{theorem}
$x^2 + iy^2 + z^2 + iw^2$ is universal over $\mathbb{Z}[i]$
\end{theorem}

\begin{proof}
Since any $z \in \mathbb{Z}[i]$ can be expressed as $z = i^sz'$ where all primes contained in $z'$ are from $\Pi'$ it is sufficient to show that $x^2 + iy^2 + z^2 + iw^2$ expresses any canonical prime element. The argument that this expression will represent $z'$ follows from corollary  1. Theorem 1 guarantees this for $p \in A$, theorem 3 guarantees this for $p \in C$ and by a simple check it also holds for $(1+i)$. One should notice that multiplying the expression $x^2 + iy^2 + z^2 + iw^2$  by elements of the unit group does not change it, which completes the proof.

\end{proof}

\section{Ramanujan's Forms}

Interestingly, if we restrict the arguments $y$ and $w$ to the complex line $(1-i)t$, and $x$ and $z$ to the real line, we will obtain one of Ramanujan's forms, explicitly: $x^2 + 2y^2 + z^2 + 2w^2$, where $x,y,z,w \in \mathbb{Z}$ \cite{Ramanujan1916}. By imposing other similar restrictions on $x,y,z$ and $w$, we obtain the forms $[1, 2, 1, 8], [1, 2, 4, 2], [1, 2, 4, 8]$, raising the question regarding the connection between imaginary forms and forms over real fields.

\bibliographystyle{ieeetr}
\bibliography{references}

\begin{thebibliography}{1}

\bibitem{Conway}
J.~H. Conway, ``Universal quadratic forms and the fifteen theorem,'' {\em
  Contemporary Mathematics}, vol.~272, pp.~23--26, 2000.

\bibitem{Mordell1920}
L.~Mordell, ``On the representation of algebraic integers as a sum of four
  squares,'' {\em Proceedings of the Cambridge Philosophical Society}, vol.~20,
  pp.~250--256, 1920.

\bibitem{Niven1940}
I.~Niven, ``Integers of quadratic fields as sums of squares,'' {\em Trans.
  Amer. Math. Soc.}, vol.~48, pp.~405--417, 1940.

\bibitem{Wai-kiuChan1996}
W.-k. Chan, M.-H. Kim, and S.~Raghavan, ``Ternary universal integral quadratic
  forms over real quadratic fields,'' {\em Japan. J. Math.}, vol.~22, no.~2,
  1996.

\bibitem{PeteL.Clark}
P.~L. Clark, J.~Hicks, K.~Thompson, and N.~Walters, ``Gonii: Universal
  quaternary quadratic forms,'' {\em Integers}, vol.~12, p.~A50, 2012.

\bibitem{Dickson1929}
L.~Dickson, ``Universal quadratic forms,'' {\em Trans. Amer. Math. Soc.},
  vol.~31, pp.~164--189, 1929.

\bibitem{HarryPollard1985}
H.~Pollard and H.~G. Diamond, {\em The Theory of Algebraic Numbers}.
\newblock Dover Publications Inc., 1985.

\bibitem{Ramanujan1916}
S.~Ramanujan, ``On the expression of a number in the from $ax^2 + by^2 + cz^2 +
  dv^2$,'' {\em Proccedings of the Cambridge Philosophical Society}, vol.~19,
  pp.~11--21, 1917.

\end{thebibliography}

\end{document}